\documentclass[11pt]{article}

\usepackage[margin=1in]{geometry}
\usepackage{amsmath,amsthm,amssymb,amsfonts,mathtools}
\usepackage{microtype}
\usepackage{enumitem}
\usepackage{hyperref}

\hypersetup{
  colorlinks=true,
  linkcolor=blue,
  citecolor=blue,
  urlcolor=blue
}

\newtheorem{theorem}{Theorem}[section]
\newtheorem{proposition}[theorem]{Proposition}
\newtheorem{lemma}[theorem]{Lemma}
\newtheorem{corollary}[theorem]{Corollary}
\newtheorem{definition}[theorem]{Definition}
\newtheorem{remark}[theorem]{Remark}
\newtheorem{example}[theorem]{Example}

\newcommand{\OK}{\mathcal O_K}
\newcommand{\OKS}{\mathcal O_{K,\Sigma}}
\newcommand{\Norm}{\mathrm N}
\newcommand{\RePart}{\operatorname{Re}}

\title{Dirichlet Series and Stirling-Type Asymptotics\\
for Generalized Legendre Factorials}
\author{Brian Diaz \and Pascal Normanyo}
\date{}

\begin{document}

\maketitle

\begin{abstract}
Let \(K\) be a fixed number field, let \(\Sigma\) be a finite set of
nonzero prime ideals of \(\mathcal O_K\), and let \(f\) be a positive
integer-valued function on the prime ideals outside \(\Sigma\).
We study the ideal-valued factorial defined by
\[
v_{\mathfrak p}(n!_{K,f,\Sigma})
=
\sum_{k\ge 0}
\left\lfloor
\frac{n}{f(\mathfrak p)(\Norm\mathfrak p)^k}
\right\rfloor .
\]
Assume that
\[
f(\mathfrak p)
=
c\,\Norm\mathfrak p
+
O\!\left((\Norm\mathfrak p)^{1-\delta}\right)
\qquad (c>0,\ \delta>0).
\]
We derive a Dirichlet series for the logarithmic increments and compare
its local prime-power sequence with the ordinary prime-ideal von
Mangoldt sequence.  A prime-ideal theorem and a Dirichlet hyperbola
argument then give
\[
\log \Norm_\Sigma(n!_{K,f,\Sigma})
=
\frac1c\,n\log n
+
C_{K,f,\Sigma}n
+
O\!\left(
n e^{-a\sqrt{\log n}}
\right)
\]
for some \(a>0\).  The linear coefficient is explicit:
\[
C_{K,f,\Sigma}
=
\frac{\gamma+\kappa_{K,\Sigma}-\log c-1}{c}
+
J_{K,f,\Sigma}(1).
\]
Here \(\kappa_{K,\Sigma}\) is the constant term of
\(-\zeta'_{K,\Sigma}/\zeta_{K,\Sigma}\) at \(s=1\), and
\(J_{K,f,\Sigma}\) is an absolutely convergent prime-ideal correction
near \(s=1\).  The method applies to factorial ideals of Legendre
subsets whose local class numbers have the corresponding geometric
form.
\end{abstract}

\section{Introduction}

For a rational prime \(p\), Legendre's formula is
\[
v_p(n!)
=
\sum_{r\ge 1}
\left\lfloor\frac{n}{p^r}\right\rfloor.
\]
This formula motivates generalized factorials whose local valuations
are defined by related floor sums.  Such constructions occur in the
theory of Bhargava factorials and in the study of Legendre subsets of
discrete valuation rings; see \cite{Bhargava,EvrardFares}.  In earlier
work, Diaz studied asymptotics for a class of Legendre formulas over
the rational primes \cite{DiazLegendre}.

Let \(K\) be a number field, let \(\Sigma\) be a finite set of nonzero
prime ideals of \(\OK\), and write
\[
\OKS
=
\{x\in K:
v_{\mathfrak p}(x)\ge 0
\text{ for every }\mathfrak p\notin\Sigma\}.
\]
For an admissible function \(f\), we define a factorial ideal by
\[
v_{\mathfrak p}(n!_{K,f,\Sigma})
=
\sum_{k\ge 0}
\left\lfloor
\frac{n}{f(\mathfrak p)(\Norm\mathfrak p)^k}
\right\rfloor .
\]
The ordinary factorial corresponds to
\(K=\mathbf Q\), \(\Sigma=\varnothing\), and \(f(p)=p\).

The logarithmic increment has Dirichlet series
\[
D_{K,f,\Sigma}(s)
=
\zeta(s)H_{K,f,\Sigma}(s),
\]
where
\[
H_{K,f,\Sigma}(s)
=
\sum_{\mathfrak p\notin\Sigma}
\frac{\log\Norm\mathfrak p}
{f(\mathfrak p)^s
 (1-(\Norm\mathfrak p)^{-s})}.
\]
If
\[
f(\mathfrak p)
=
c\,\Norm\mathfrak p
+
O\!\left((\Norm\mathfrak p)^{1-\delta}\right),
\]
then
\[
H_{K,f,\Sigma}(s)
=
c^{-s}
\left(
-\frac{\zeta'_{K,\Sigma}}{\zeta_{K,\Sigma}}(s)
\right)
+
J_{K,f,\Sigma}(s),
\]
with \(J_{K,f,\Sigma}\) holomorphic in a fixed half-plane containing
\(s=1\).  This gives the correct double pole of \(D_{K,f,\Sigma}\), but
a Laurent expansion alone does not justify a secondary asymptotic
term; counterexamples in Tauberian theory show that quantitative
information is essential \cite{PierceTauberian}.

The additional arithmetic input used here is a prime-ideal theorem for
the generalized local sequence
\[
f(\mathfrak p)(\Norm\mathfrak p)^k.
\]
We prove that its generalized Chebyshev function differs from the
ordinary prime-ideal Chebyshev function only by a power-saving term.
A Dirichlet hyperbola decomposition then converts that one-dimensional
prime-ideal estimate into the desired two-term Stirling formula.

\begin{theorem}\label{thm:main}
Let \(K\) be a fixed number field, let \(\Sigma\) be a finite set of
nonzero prime ideals of \(\OK\), and let
\[
f:\{\mathfrak p:\mathfrak p\notin\Sigma\}
\longrightarrow \mathbf Z_{\ge 1}
\]
satisfy
\[
f(\mathfrak p)
=
c\,\Norm\mathfrak p
+
O\!\left((\Norm\mathfrak p)^{1-\delta}\right)
\]
for constants \(c>0\) and \(\delta>0\).  Then there is a constant
\(a=a(K,\Sigma,f)>0\) such that
\[
\log\Norm_\Sigma(n!_{K,f,\Sigma})
=
\frac1c\,n\log n
+
C_{K,f,\Sigma}n
+
O_{K,\Sigma,f}\!\left(
n e^{-a\sqrt{\log n}}
\right),
\]
where
\[
C_{K,f,\Sigma}
=
\frac{\gamma+\kappa_{K,\Sigma}-\log c-1}{c}
+
J_{K,f,\Sigma}(1).
\]
The constants \(\kappa_{K,\Sigma}\) and
\(J_{K,f,\Sigma}(1)\) are defined in
Sections~\ref{sec:zeta-comparison} and \ref{sec:constant}.
\end{theorem}

Because \(K\) is fixed, a possible exceptional real zero of
\(\zeta_K\) is absorbed into the displayed error term.  A statement
uniform over varying number fields would require the exceptional-zero
contribution to be tracked separately.

\section{Arithmetic setup and factorial ideals}

Throughout, \(\Norm\mathfrak p=|\OK/\mathfrak p|\).  Every nonzero
integral ideal \(\mathfrak a\) of \(\OKS\) has a unique factorization
\[
\mathfrak a
=
\prod_{\mathfrak p\notin\Sigma}
(\mathfrak p\OKS)^{v_{\mathfrak p}(\mathfrak a)}
\]
with finite support.  Define
\[
\Norm_\Sigma(\mathfrak a)
=
\prod_{\mathfrak p\notin\Sigma}
(\Norm\mathfrak p)^{v_{\mathfrak p}(\mathfrak a)}.
\]

\begin{definition}
The function \(f\) is \emph{admissible} if there exist \(c>0\) and
\(\delta>0\) such that
\[
f(\mathfrak p)
=
c\,\Norm\mathfrak p
+
O\!\left((\Norm\mathfrak p)^{1-\delta}\right).
\]
\end{definition}

Admissibility implies \(f(\mathfrak p)\gg\Norm\mathfrak p\) outside a
finite set.  Hence, for fixed \(n\), only finitely many prime ideals
contribute to the following definition.

\begin{definition}
Set \(0!_{K,f,\Sigma}=\OKS\).  For \(n\ge 1\), define the integral ideal
\(n!_{K,f,\Sigma}\) by
\[
v_{\mathfrak p}(n!_{K,f,\Sigma})
=
\sum_{k\ge 0}
\left\lfloor
\frac{n}{f(\mathfrak p)(\Norm\mathfrak p)^k}
\right\rfloor
\qquad(\mathfrak p\notin\Sigma).
\]
\end{definition}

Define
\[
B_{K,f,\Sigma}(n)
=
\log\Norm_\Sigma(n!_{K,f,\Sigma})
-
\log\Norm_\Sigma((n-1)!_{K,f,\Sigma}).
\]

\begin{proposition}\label{prop:increment}
For \(n\ge1\),
\[
B_{K,f,\Sigma}(n)
=
\sum_{\mathfrak p\notin\Sigma}
\log\Norm\mathfrak p
\sum_{k\ge0}
\mathbf 1_{
f(\mathfrak p)(\Norm\mathfrak p)^k\mid n}.
\]
In particular, \(B_{K,f,\Sigma}(n)\ge0\), and
\[
\log\Norm_\Sigma(n!_{K,f,\Sigma})
=
\sum_{m\le n}B_{K,f,\Sigma}(m).
\]
\end{proposition}

\begin{proof}
Use
\[
\left\lfloor\frac{n}{d}\right\rfloor
-
\left\lfloor\frac{n-1}{d}\right\rfloor
=
\mathbf 1_{d\mid n},
\]
then sum the valuation differences with weights
\(\log\Norm\mathfrak p\).
\end{proof}

\section{The increment Dirichlet series}

Define a nonnegative arithmetic function
\[
h_{K,f,\Sigma}(m)
=
\sum_{\substack{\mathfrak p\notin\Sigma,\ k\ge0\\
f(\mathfrak p)(\Norm\mathfrak p)^k=m}}
\log\Norm\mathfrak p.
\]
Then Proposition~\ref{prop:increment} says
\[
B_{K,f,\Sigma}=\mathbf 1*h_{K,f,\Sigma}
\]
as a Dirichlet convolution.

\begin{proposition}\label{prop:dirichlet}
For \(\RePart(s)>1\),
\[
H_{K,f,\Sigma}(s)
:=
\sum_{m\ge1}
\frac{h_{K,f,\Sigma}(m)}{m^s}
=
\sum_{\mathfrak p\notin\Sigma}
\frac{\log\Norm\mathfrak p}
{f(\mathfrak p)^s
 (1-(\Norm\mathfrak p)^{-s})},
\]
and
\[
D_{K,f,\Sigma}(s)
:=
\sum_{n\ge1}
\frac{B_{K,f,\Sigma}(n)}{n^s}
=
\zeta(s)H_{K,f,\Sigma}(s).
\]
All series converge absolutely in this half-plane.
\end{proposition}

\begin{proof}
Admissibility gives \(f(\mathfrak p)\gg\Norm\mathfrak p\), apart from
finitely many primes.  Absolute convergence follows from the standard
prime-ideal series for
\(-\zeta'_{K,\Sigma}/\zeta_{K,\Sigma}\).  The displayed identities
follow by grouping the geometric progression in \(k\) and using the
Dirichlet-convolution identity from
Proposition~\ref{prop:increment}.
\end{proof}

\section{Comparison with the Dedekind zeta function}
\label{sec:zeta-comparison}

The partial Dedekind zeta function is
\[
\zeta_{K,\Sigma}(s)
=
\prod_{\mathfrak p\notin\Sigma}
(1-(\Norm\mathfrak p)^{-s})^{-1}.
\]
For \(\RePart(s)>1\),
\[
-\frac{\zeta'_{K,\Sigma}}{\zeta_{K,\Sigma}}(s)
=
\sum_{\mathfrak p\notin\Sigma}
\frac{\log\Norm\mathfrak p}
{(\Norm\mathfrak p)^s
 (1-(\Norm\mathfrak p)^{-s})}.
\]
Near \(s=1\), write
\[
-\frac{\zeta'_{K,\Sigma}}{\zeta_{K,\Sigma}}(s)
=
\frac1{s-1}
+
\kappa_{K,\Sigma}
+
O(s-1).
\]

Put
\[
\eta=\min\{\delta,\tfrac12\}.
\]

\begin{proposition}\label{prop:J}
The function
\[
J_{K,f,\Sigma}(s)
=
H_{K,f,\Sigma}(s)
-
c^{-s}
\left(
-\frac{\zeta'_{K,\Sigma}}{\zeta_{K,\Sigma}}(s)
\right)
\]
is holomorphic for
\[
\RePart(s)>1-\eta.
\]
In particular,
\[
H_{K,f,\Sigma}(s)
=
\frac{1/c}{s-1}
+
\beta_{K,f,\Sigma}
+
O(s-1)
\]
near \(s=1\), where
\[
\beta_{K,f,\Sigma}
=
\frac{\kappa_{K,\Sigma}-\log c}{c}
+
J_{K,f,\Sigma}(1).
\]
\end{proposition}

\begin{proof}
Write \(q=\Norm\mathfrak p\) and
\[
f(\mathfrak p)=cq(1+\varepsilon_{\mathfrak p}),
\qquad
\varepsilon_{\mathfrak p}=O(q^{-\delta}).
\]
On a compact subset of \(\RePart(s)>1-\eta\), the factors
\((1-q^{-s})^{-1}\) are uniformly bounded, and
\[
(1+\varepsilon_{\mathfrak p})^{-s}-1
=
O(q^{-\delta})
\]
uniformly.  Hence the \(\mathfrak p\)-th difference is
\[
O\!\left(
\frac{\log q}{q^{\RePart(s)+\delta}}
\right).
\]
The resulting prime-ideal series converges locally uniformly because
\(\RePart(s)+\delta>1\).  This proves holomorphy.  The Laurent
coefficient follows by multiplying
\[
c^{-s}
=
\frac1c
\left(
1-(s-1)\log c+O((s-1)^2)
\right)
\]
with the Laurent expansion of the logarithmic derivative.
\end{proof}

\section{A deformed prime-ideal theorem}

Define
\[
\Psi_{K,f,\Sigma}(x)
=
\sum_{m\le x}h_{K,f,\Sigma}(m)
=
\sum_{\mathfrak p\notin\Sigma}
\log\Norm\mathfrak p
\sum_{\substack{k\ge0\\
f(\mathfrak p)(\Norm\mathfrak p)^k\le x}}
1.
\]
Also define the partial prime-ideal Chebyshev function
\[
\psi_{K,\Sigma}(x)
=
\sum_{\substack{\mathfrak p\notin\Sigma,\ r\ge1\\
(\Norm\mathfrak p)^r\le x}}
\log\Norm\mathfrak p.
\]

For fixed \(K\) and \(\Sigma\), the prime ideal theorem gives
\[
\psi_{K,\Sigma}(x)
=
x+
O_{K,\Sigma}\!\left(
x e^{-a_0\sqrt{\log x}}
\right)
\qquad(x\ge3)
\]
for some \(a_0>0\); this follows, for example, from the effective
Chebotarev theorem of Lagarias and Odlyzko
\cite{LagariasOdlyzko}.  A possible exceptional zero causes a term
\(x^\beta\), but for fixed \(K\) this is absorbed by the displayed
error after reducing \(a_0\).

\begin{lemma}\label{lem:interval}
Let \(d=[K:\mathbf Q]\).  For \(2\le u<v\),
\[
\sum_{\substack{\mathfrak p\\
u<\Norm\mathfrak p\le v}}
\log\Norm\mathfrak p
\ll_K
(v-u+1)\log v.
\]
\end{lemma}

\begin{proof}
For each positive integer \(m\), at most \(d\) prime ideals of \(K\)
have norm \(m\).  Summing the trivial bound
\(\log\Norm\mathfrak p\le\log v\) proves the claim.
\end{proof}

\begin{proposition}[Deformed prime-ideal theorem]
\label{prop:deformed-pit}
There is \(a_1=a_1(K,\Sigma,f)>0\) such that
\[
\Psi_{K,f,\Sigma}(x)
=
\frac{x}{c}
+
O_{K,\Sigma,f}\!\left(
x e^{-a_1\sqrt{\log x}}
\right)
\qquad(x\ge3).
\]
More precisely, if
\(\eta=\min\{\delta,\tfrac12\}\), then
\[
\Psi_{K,f,\Sigma}(x)
-
\psi_{K,\Sigma}(x/c)
\ll_{K,\Sigma,f}
x^{1-\eta}(\log x)^2.
\]
\end{proposition}

\begin{proof}
Write \(q=\Norm\mathfrak p\) and \(r=k+1\).  The model denominator
corresponding to
\(f(\mathfrak p)q^{r-1}\) is \(cq^r\).  Outside a fixed finite set of
prime ideals,
\[
f(\mathfrak p)q^{r-1}
=
cq^r(1+O(q^{-\delta})).
\]
If the two indicators
\[
\mathbf 1_{f(\mathfrak p)q^{r-1}\le x}
\quad\text{and}\quad
\mathbf 1_{cq^r\le x}
\]
differ, then
\[
\left|x-cq^r\right|
\ll q^{r-\delta}.
\]
Put \(Y_r=(x/c)^{1/r}\).  The preceding inequality implies
\[
|q-Y_r|
\ll
Y_r^{1-\eta},
\]
after enlarging the implied constant.  Thus, by
Lemma~\ref{lem:interval}, the total weighted contribution for a fixed
\(r\) is
\[
\ll_K
\bigl(Y_r^{1-\eta}+1\bigr)\log x.
\]

Only \(O(\log x)\) values of \(r\) occur.  The term \(r=1\) contributes
\(O(x^{1-\eta}\log x)\).  For \(r\ge2\), one has
\(Y_r\ll x^{1/2}\), so the total contribution of these \(r\) is
\[
O\!\left(
x^{(1-\eta)/2}(\log x)^2
+
(\log x)^2
\right),
\]
which is dominated by
\(O(x^{1-\eta}(\log x)^2)\).  The fixed exceptional set of prime ideals
contributes only \(O(\log x)\).  Therefore
\[
\Psi_{K,f,\Sigma}(x)
=
\psi_{K,\Sigma}(x/c)
+
O\!\left(
x^{1-\eta}(\log x)^2
\right).
\]
Combining this with the prime ideal theorem and reducing the
exponential constant proves the first assertion.
\end{proof}

\section{The secondary constant}
\label{sec:constant}

\begin{proposition}[Weighted Mertens formula]
\label{prop:mertens}
There is \(a_2>0\) such that
\[
\sum_{m\le y}
\frac{h_{K,f,\Sigma}(m)}{m}
=
\frac1c\log y
+
\beta_{K,f,\Sigma}
+
O_{K,\Sigma,f}\!\left(
e^{-a_2\sqrt{\log y}}
\right),
\]
where
\[
\beta_{K,f,\Sigma}
=
\frac{\kappa_{K,\Sigma}-\log c}{c}
+
J_{K,f,\Sigma}(1).
\]
\end{proposition}

\begin{proof}
Put
\[
E_f(t)
=
\Psi_{K,f,\Sigma}(t)-\frac{t}{c}.
\]
Proposition~\ref{prop:deformed-pit} gives
\[
E_f(t)
\ll
t e^{-a_1\sqrt{\log t}}.
\]
Partial summation yields
\begin{align*}
\sum_{m\le y}
\frac{h_{K,f,\Sigma}(m)}{m}
&=
\frac{\Psi_{K,f,\Sigma}(y)}{y}
+
\int_1^y
\frac{\Psi_{K,f,\Sigma}(t)}{t^2}\,dt\\
&=
\frac1c\log y
+
\frac1c
+
\frac{E_f(y)}{y}
+
\int_1^y\frac{E_f(t)}{t^2}\,dt.
\end{align*}
The integral
\(\int_1^\infty E_f(t)t^{-2}\,dt\) converges, and its tail is
\(O(e^{-a_2\sqrt{\log y}})\) after reducing \(a_2\).

It remains to identify the constant.  For \(\RePart(s)>1\),
partial summation gives
\[
H_{K,f,\Sigma}(s)
=
s\int_1^\infty
\Psi_{K,f,\Sigma}(t)t^{-s-1}\,dt.
\]
Substituting
\(\Psi_{K,f,\Sigma}(t)=t/c+E_f(t)\) and letting \(s\to1^+\), the
constant term of \(H_{K,f,\Sigma}\) is
\[
\frac1c+\int_1^\infty\frac{E_f(t)}{t^2}\,dt.
\]
By Proposition~\ref{prop:J}, this constant term is precisely
\(\beta_{K,f,\Sigma}\).
\end{proof}

\section{Dirichlet hyperbola and proof of the main theorem}

Let
\[
A_{K,f,\Sigma}(x)
=
\sum_{n\le x}B_{K,f,\Sigma}(n).
\]
Since \(B_{K,f,\Sigma}=\mathbf1*h_{K,f,\Sigma}\),
\[
A_{K,f,\Sigma}(x)
=
\sum_{dm\le x}h_{K,f,\Sigma}(d).
\]

\begin{proposition}\label{prop:hyperbola}
Let
\[
\Psi(x)=ax+O\!\left(xe^{-b\sqrt{\log x}}\right)
\]
be the summatory function of a nonnegative arithmetic function \(h\),
and suppose
\[
\sum_{n\le y}\frac{h(n)}n
=
a\log y+\beta
+
O\!\left(e^{-b'\sqrt{\log y}}\right).
\]
Then, for some \(b''>0\),
\[
\sum_{dm\le x}h(d)
=
ax\log x
+
(\beta+a\gamma-a)x
+
O\!\left(
xe^{-b''\sqrt{\log x}}
\right).
\]
\end{proposition}

\begin{proof}
Take \(Y=\sqrt{x}\) and \(M=\lfloor x/Y\rfloor\).  Dirichlet's
hyperbola decomposition gives
\begin{align*}
\sum_{dm\le x}h(d)
&=
\sum_{d\le Y}h(d)\left\lfloor\frac{x}{d}\right\rfloor
+
\sum_{m\le M}\Psi(x/m)
-
M\Psi(Y).
\end{align*}

The first term is
\[
x\sum_{d\le Y}\frac{h(d)}d+O(\Psi(Y))
=
ax\log Y+\beta x
+
O\!\left(
xe^{-b_1\sqrt{\log x}}+Y
\right).
\]
For the second term, uniformly for \(m\le M\), one has \(x/m\ge Y\);
therefore
\[
\sum_{m\le M}\Psi(x/m)
=
ax\sum_{m\le M}\frac1m
+
O\!\left(
x e^{-b_2\sqrt{\log Y}}
\sum_{m\le M}\frac1m
\right).
\]
After reducing the exponential constant,
\[
\sum_{m\le M}\Psi(x/m)
=
ax(\log M+\gamma)
+
O\!\left(
xe^{-b_3\sqrt{\log x}}+Y
\right).
\]
Finally,
\[
M\Psi(Y)
=
aMY
+
O\!\left(
MYe^{-b\sqrt{\log Y}}
\right)
=
ax
+
O\!\left(
xe^{-b_4\sqrt{\log x}}+Y
\right).
\]
Since
\[
\log Y+\log M
=
\log x+O(Y/x)
\]
and \(Y=x^{1/2}\) is absorbed by the exponential error, the result
follows.
\end{proof}

\begin{proof}[Proof of Theorem~\ref{thm:main}]
Apply Proposition~\ref{prop:hyperbola} with
\[
h=h_{K,f,\Sigma},
\qquad
a=\frac1c,
\qquad
\beta=\beta_{K,f,\Sigma}.
\]
Propositions~\ref{prop:deformed-pit} and \ref{prop:mertens} verify the
hypotheses.  Hence
\[
A_{K,f,\Sigma}(x)
=
\frac1c x\log x
+
\left(
\beta_{K,f,\Sigma}
+\frac{\gamma-1}{c}
\right)x
+
O\!\left(
xe^{-a\sqrt{\log x}}
\right).
\]
Using
\[
\beta_{K,f,\Sigma}
=
\frac{\kappa_{K,\Sigma}-\log c}{c}
+
J_{K,f,\Sigma}(1)
\]
gives
\[
C_{K,f,\Sigma}
=
\frac{\gamma+\kappa_{K,\Sigma}-\log c-1}{c}
+
J_{K,f,\Sigma}(1).
\]
For integer \(n\),
\[
A_{K,f,\Sigma}(n)
=
\log\Norm_\Sigma(n!_{K,f,\Sigma})
\]
by Proposition~\ref{prop:increment}.
\end{proof}

\section{Legendre subsets and examples}

Evrard and Fares \cite{EvrardFares} study subsets of a discrete
valuation ring whose factorial valuations satisfy
\[
v(n!_T)
=
\sum_{i\ge1}
\left\lfloor\frac{n}{q_i}\right\rfloor,
\]
where \(q_i\) is the number of residue classes of the subset modulo the
\(i\)-th power of the maximal ideal.

\begin{corollary}\label{cor:legendre}
Let \(R=\OKS\), and let \(T\subseteq R\) be a subset whose Bhargava
factorial ideals satisfy, for every \(\mathfrak p\notin\Sigma\),
\[
v_{\mathfrak p}(n!_{R,T})
=
\sum_{i\ge1}
\left\lfloor
\frac{n}{q_{\mathfrak p,i}}
\right\rfloor.
\]
Assume
\[
q_{\mathfrak p,i}
=
f(\mathfrak p)(\Norm\mathfrak p)^{i-1},
\qquad i\ge1,
\]
for an admissible function \(f\).  Then
\[
\log\Norm_\Sigma(n!_{R,T})
=
\frac1c n\log n
+
C_{K,f,\Sigma}n
+
O\!\left(
ne^{-a\sqrt{\log n}}
\right).
\]
\end{corollary}

\begin{proof}
Set \(k=i-1\).  The local valuation formula is exactly that of
\(n!_{K,f,\Sigma}\), so Theorem~\ref{thm:main} applies.
\end{proof}

\begin{remark}
The number-field qualification is essential.  A Dedekind domain
finitely generated as a \(\mathbf Z\)-algebra can have positive
characteristic; for example, \(\mathbf F_p[t]\) is not a ring of
\(S\)-integers in a number field.  The present result concerns the
characteristic-zero setting.
\end{remark}

\begin{example}[The ordinary factorial]
Let \(K=\mathbf Q\), \(\Sigma=\varnothing\), and \(f(p)=p\).  Then
\(c=1\), \(J(s)=0\), and
\[
-\frac{\zeta'}{\zeta}(s)
=
\frac1{s-1}-\gamma+O(s-1),
\]
so \(\kappa_{\mathbf Q}=-\gamma\).  Therefore
\[
C_{\mathbf Q,f,\varnothing}=-1,
\]
and Theorem~\ref{thm:main} gives
\[
\log(n!)
=
n\log n-n
+
O\!\left(
ne^{-a\sqrt{\log n}}
\right),
\]
consistent with Stirling's formula.
\end{example}

\begin{example}
Let \(K=\mathbf Q\), \(\Sigma=\varnothing\), and \(f(p)=p-1\).  Then
\(c=1\), and Theorem~\ref{thm:main} gives
\[
\log\Norm(n!_{\mathbf Q,f,\varnothing})
=
n\log n
+
C_f n
+
O\!\left(
ne^{-a\sqrt{\log n}}
\right).
\]
This recovers the shape of the asymptotic studied in
\cite{DiazLegendre}, while the Dirichlet-series formula identifies the
linear coefficient.
\end{example}

\begin{remark}
Changing \(f\) at finitely many prime ideals, or enlarging \(\Sigma\)
by finitely many prime ideals, leaves the leading coefficient \(1/c\)
unchanged but alters \(J_{K,f,\Sigma}(1)\), and hence the linear
coefficient.
\end{remark}

\end{document}